\documentclass[11pt]{amsart}
\usepackage{amsmath}
\usepackage{amssymb}
\usepackage{latexsym}
\usepackage{graphicx}

\newtheorem{theorem}{Theorem}
\newtheorem{corollary}{Corollary}
\newtheorem{lemma}{Lemma}
\theoremstyle{remark}
\newtheorem{remark}{Remark}

\begin{document}

\title[quasi-convolution of analytic functions]{\large quasi-convolution of analytic functions with applications}

\author[K. O. Babalola]{K. O. BABALOLA}

\begin{abstract}
In this paper we define a new concept of quasi-convolution for analytic functions normalized by $f(0)=0$ and $f^\prime(0)=1$ in the unit disk $E=\{z\in \mathbb{C}\colon |z|<1\}$. We apply this new approach to study the closure properties of a certain class of analytic and univalent functions under some families of (known and new) integral operators.
\end{abstract}



\maketitle

\section{Introduction}
Let $\mathbb{A}$ denote the class of functions:
\begin{equation}
f(z)=z+a_2z^2+...\, \label{1}
\end{equation}
which are analytic in the unit disk $E=\{z\in \mathbb{C}\colon
|z|<1\}$. In \cite{TOO}, Opoola introduced the subclass
$T_n^\alpha(\beta)$ consisting of functions $f \in A$ which
satisfy:
\begin{equation}
Re \frac{D^nf(z)^\alpha}{\alpha^nz^\alpha}>\beta\, \label{2}
\end{equation}
\\where $\alpha>0$ is real, $0\le\beta<1$, $D^n(n\in\mathbb{N})$ is the Salagean derivative operator defined as:
$D^nf(z)=D(D^{n-1}f(z))=z[D^{n-1}f(z)]^{\prime}$ with
$D^0f(z)=f(z)$ and powers in ~(\ref{2}) meaning principal
determinations only. The geometric condition ~(\ref{2}) slightly
modifies the one given originally in \cite{TOO} (see \cite{BO}). If $\beta=0$, we have the class $B_n(\alpha)$ studied by Abdulhalim in \cite{SA}.

Let $g(z)=z+b_2z^2+... \in\mathbb{A}$. The convolution (or Hadamard
product) of $f(z)$ and $g(z)$ (written as $(f*g)(z)$) is defined as
\[
(f*g)(z)=z+\sum_{k=2}^\infty a_kb_kz^k.
\]

The above concept has proved very resourceful in dealing with certain problems of the theory of analytic and univalent functions, especially closure of families of functions under certain transformations (see\cite{NN}). It is natural, therefore, to desire to investigate the convolution properties of classes of functions. However, readers familiar with studies in Bazilevic functions of type $\alpha$ (see \cite{SA, KOB, BO, BO2, THM, TOO, RS, KY}) would appreciate the challenges posed by the index $\alpha$ in some characterizations of those functions, particularly when $\alpha\neq 1$ or is generally non-integers. Of note, in particular, is in their convolution. Perhaps it is the reason the convolution problem for the various families of such functions has not been addressed, or that no single paper has appeared treating it, especially in the case $\alpha$ is not an integer, as far as the present author is aware! To begin to look at the problem we propose an idea of quasi-convolution as follows: Let us recall that the concept  of convolution actually arose from the integral
\[
h(r^2e^{i\theta})=(f*g)(r^2e^{i\theta})=\frac{1}{2\pi}\int_0^{2\pi}f(re^{i(\theta-t)})g(re^{it})dt,\;\;\;r<1
\]
and that the integral convolution is defined by
\[
H(z)=\int_0^z\xi^{-1}h(\xi) d\xi,\;\;\;|\xi|<1
\]
(see \cite{PLD}). Now, since for $\alpha>0$, we can write $f(z)^\alpha$ and $g(z)^\alpha$ as $f(z)^\alpha=z^\alpha+A_2(\alpha)z^{\alpha+1}+...$ and $g(z)^\alpha=z^\alpha+B_2(\alpha)z^{\alpha+1}+...$ (where $A_k(\alpha)$, $B_k(\alpha)$ respectively depend on the coeficients $a_k$ of $f(z)$ and $b_k$ of $g(z)$, and $\alpha$), we can define the following integrals.
\begin{equation}
\phi(r^2e^{i\theta})^\alpha=(f^\alpha*g^\alpha)(r^2e^{i\theta})=\frac{1}{2\pi}\int_0^{2\pi}f(re^{i(\theta-t)})^\alpha g(re^{it})^\alpha dt,\;\;\;r<1\, \label{3}
\end{equation}
and 
\begin{equation}
\Phi(z)^\alpha=\alpha\int_0^z\xi^{-1}\phi(\xi)^\alpha d\xi,\;\;\;|\xi|<1.\, \label{4}
\end{equation}

By the integrals ~(\ref{3}) and ~(\ref{4}), we define the following new concepts:

\medskip

 {\sc Definition 1.} Let $f$, $g\in\mathbb{A}$. Let $\alpha>0$ be real. We define the quasi-convolution of $f(z)$ and $g(z)$ (denoted by $\phi(z)^\alpha=(f^\alpha*g^\alpha)(z)$) as
\begin{equation}
\phi(z)^\alpha=(f^\alpha*g^\alpha)(z)=z^\alpha+\sum_{k=2}^\infty A_k(\alpha)B_k(\alpha)z^{\alpha+k-1},\, \label{5}
\end{equation}
and the integral quasi-convolution is correspondingly defined by
\[
\Phi(z)^\alpha=\alpha\int_0^z\xi^{-1}\phi(\xi)^\alpha d\xi,\;\;\;|\xi|<1.
\]

This is my thinking! The justification for it lies in some very interesting applications, which we provide in Section 4. Earlier works involving quasi-convolution of analytic functtions can be found in the literatures \cite{KOS, OS}. For $\alpha=1$, we have the well known convolution (Hadamard product) of analytic functions.

Throughout this paper, $\phi(z)$ will be defined by the integral ~(\ref{3}) having series expansion ~(\ref{5}) and we will be investigating the class $T_n^\alpha (\beta)$ under $\phi(z)$, for two cases, namely, (i) $g(z)$ convex, $f\in T_n^\alpha(\beta)$, and (ii) $f$, $g\in T_n^\alpha(\beta)$. Our results are contained in Section 3, followed by some nice applications in Section 4. In the next section we give some preliminary lemmas and notes.

\section{Some Lemmas and Notes}
\medskip

 {\sc Definition 2.} Let $u=u_1+u_2i$, $v=v_1+v_2i$. Define $\Psi$ as the set of functions $\psi(u,v):\mathbb{C}\times\mathbb{C}\to\mathbb{C}$ satisfying:
 
{\rm(a)} $\psi(u,v)$ is continuous in a domain $\Omega$ of $\mathbb{C}\times\mathbb{C}$,

{\rm(b)} $(1,0)\in\Omega$ and Re$\psi(1,0)>0$,

{\rm(c)} Re$\psi(u_2i, v_1)\leq 0$ when $(u_2i, v_1)\in\Omega$ and $v_1\leq -\tfrac{1}{2}(1+u_2^2)$.

Several examples of members of the set $\Psi$ have been mentioned in \cite{BO,MM} and \cite[p.27]{SS}. We shall need the following member:
\[
\psi(u,v)=\frac{1}{2}+\frac{v}{\alpha(1+u)} 
\]
where $0<\alpha\leq 1$ and $\Omega=[\mathbb{C}-\{-1\}]\times\mathbb{C}$. To see this, observe that $\psi$ is continous on $\Omega$, $(1,0)\in\Omega$ and Re $\psi(1,0)>0$, and furthermore Re $\psi(u_2i,v_1)=\tfrac{1}{2}+\tfrac{v_1}{\alpha(1+u_2^2)}$. Then if $v_1\leq -\tfrac{1}{2}(1+u_2^2)$ we get Re $\psi(u_2i,v_1)=\tfrac{\alpha-1}{2}\leq 0$. Thus $\psi\in\Psi$.
\medskip

 {\sc Definition 3.} Let $\psi\in\Psi$ with corresponding domain $\Omega$. Define $P(\Psi)$ as the set of functions $p(z)$ given as $p(z)=1+c_1z+ c_2z^2+...$ which are regular in $E$ and satisfy:
 
{\rm(i)} $(p(z),zp^\prime(z))\in\Omega$

{\rm(ii)} Re$\psi(p(z),zp^\prime(z))>0$ when $z\in E$.

More general concepts were discussed in \cite{BO,MM,SS}. 

\begin{lemma}[\cite{BO,MM,SS}]
Let $p\in P(\Psi)$. Then Re $p(z)>0$.
\end{lemma}

\begin{lemma}[\cite{FMA}]
If $p(z)$ is analytic in $E$, $p(0)=1$ and Re $p(z)>1/2$, $z\in E$, then for \underline{any} function $q(z)$ analytic in $E$, the convolution $p*q$ takes its values in the convex hull of $q(E)$.
\end{lemma}
\medskip

 {\sc Definition 4.(\cite{FMA})} An infinite sequence $a_0,a_1,...,a_k,...$ of nonnegative numbers is said to be a {\em convex null sequence} if $a_k\rightarrow 0$ as $k\rightarrow\infty$ and $a_0-a_1\geq a_1-a_2\geq ...\geq a_k-a_{k+1}\geq...\geq 0$.
\begin{lemma}[\cite{FMA}]
Let $\{c_k\}_{k=0}^\infty$ be a convex null sequence. Then the function $p(z)=c_0/2+c_1z+c_2z^2+...$, $z\in E$, is analytic in $E$ and Re $p(z)>0$.
\end{lemma}

If $a$, $b$ are nonzero positive real numbers such that $a>b$, it can be shown by simple inductive process that 
\[
(a-b)^m\leq a^m-b^m,\;\;\;m\in\mathbb{N}.
\]

By this it can be easily seen that the infinite sequence $\left\{d_k\right\}_{k=0}^\infty$
where
\[
d_k=\frac{\alpha^m}{(\alpha+k)^m},\;\;\alpha>0,\;\;m\in\mathbb{N}
\]
is convex null. In fact it has been mentioned in \cite{BO} that the sequence preserves many geometric structures of analytic functions, particularly starlikeness, convexity and subordination. In this article we would make use of the convex null sequence $\left\{d_k\right\}_{k=0}^\infty$ in the investigation of convolution properties of functions of the class $T_n^\alpha(\beta)$. 

We now turn to the main result.

\section{Quasi-Convolution}
First we prove

\begin{theorem}
Let $0<\alpha\leq 1$. If $g(z)$ is convex, then Re $g(z)^\alpha/z^\alpha>1/2$.
\end{theorem}
\begin{proof}
It is known that if $g(z)$ is convex, then it is starlike of order $\tfrac{1}{2}$. Let $0<\alpha\leq 1$ and define
\[
p(z)=2\frac{g(z)^\alpha}{z^\alpha}-1.
\]
Then
\[
\frac{zg^\prime(z)}{g(z)}-\frac{1}{2}=\frac{1}{2}+\frac{zp^\prime(z)}{\alpha(1+p(z))}.
\]
Let $\psi(u,v)$ be defined on a domain $\Omega=[\mathbb{C}-\{-1\}]\times\mathbb{C}$ by $\psi(u,v)=\frac{1}{2}+\tfrac{v}{\alpha(1+u)}$ where $u=p(z)$ and $v=zp^\prime(z)$ and $0<\alpha\leq 1$. Thus by Lemma 1 we have Re $(zg^\prime(z)/g(z)-1/2)>0\Rightarrow$ Re $p(z)>0$ and consequently we have
\[
Re\left\{\frac{(zg^\prime(z))^\prime}{g^\prime(z)}\right\}>0\Rightarrow Re\frac{zg^\prime(z)}{g(z)}>\frac{1}{2}\Rightarrow Re\frac{g(z)^\alpha}{z^\alpha}>\frac{1}{2},
\]
which completes the proof.
\end{proof}

The above proof is adapted from \cite{KOB} for completeness. For $\alpha=1$, the result can be found in literatures (see \cite{MM} for example).

\begin{theorem}
Let $f\in T_n^\alpha(\beta)$ and $g\in C$. If $0<\alpha\leq 1$, then $\phi\in T_n^\alpha(\beta)$, that is
\[
Re\frac{D^n\phi(z)^\alpha}{\alpha^nz^\alpha}>\beta.
\]
\end{theorem}
\begin{proof}
Since $g(z)$ is convex, then for $0<\alpha\leq 1$, we have Re $g(z)^\alpha/z^\alpha>1/2$. Hence by Lemma 2, the normalized analytic function defined by
\[
\frac{D^nf(z)^\alpha}{\alpha^nz^\alpha}*\frac{g(z)^\alpha}{z^\alpha}=\frac{D^n(f(z)^\alpha*g(z)^\alpha)}{\alpha^nz^\alpha}=\frac{D^n\phi(z)^\alpha}{\alpha^nz^\alpha}
\]
takes values in the convex hull of the image of $E$ under $\tfrac{D^nf(z)^\alpha}{\alpha^nz^\alpha}$, hence
\[
Re\frac{D^n\phi(z)^\alpha}{\alpha^nz^\alpha}>\beta,
\]
that is $\phi\in T_n^\alpha(\beta)$.
\end{proof}
\begin{theorem}
Let $f\in T_n^\alpha(\beta)$ and $g\in T_m^\alpha(\lambda)$, $1/2\leq\beta+\gamma<3/2$. Then $\phi\in T_n^\alpha(\beta+\lambda-\tfrac{1}{2})$, that is
\[
Re\frac{D^n\phi(z)^\alpha}{\alpha^nz^\alpha}>\frac{2(\beta+\lambda)-1}{2}.
\]
\end{theorem}
\begin{proof}
Since the sequence $\{d_k\}_{k=0}^\infty$ is a convex null sequence, we have, by Lemma 3, Re $\varphi(z)^\alpha/z^\alpha>1/2$ where $\varphi(z)$ is defined by
\[
\frac{\varphi(z)^\alpha}{z^\alpha}=1+\sum_{k=2}^\infty\frac{\alpha^m}{(\alpha+k-1)^m}z^{k-1}.
\]
But
\[
\frac{D^mg(z)^\alpha}{\alpha^mz^\alpha}=1+\sum_{k=2}^\infty\frac{(\alpha+k-1)^m}{\alpha^m}b_k(\alpha)z^{k-1}.
\]
Hence
\begin{multline*}
\frac{D^mg(z)^\alpha}{\alpha^mz^\alpha}*\frac{\varphi(z)^\alpha}{z^\alpha}=\frac{D^m(g(z)^\alpha*\varphi(z)^\alpha)}{\alpha^mz^\alpha}\\
=1+\sum_{k=2}^\infty b_k(\alpha)z^{k-1}=\frac{g(z)^\alpha}{z^\alpha}.
\end{multline*}
Thus by Lemma 2, we have Re $g(z)^\alpha/z^\alpha>\lambda$, so that
\[
Re\left(\frac{g(z)^\alpha}{z^\alpha}-\lambda+\frac{1}{2}\right)>\frac{1}{2}
\]
By Lemma 2 again the analytic function (though not normalized) given by
\begin{multline*}
\frac{D^nf(z)^\alpha}{\alpha^nz^\alpha}*\left(\frac{g(z)^\alpha}{z^\alpha}-\lambda+\frac{1}{2}\right)\\=\frac{D^n(f(z)^\alpha*g(z)^\alpha)}{\alpha^nz^\alpha}-\lambda+\frac{1}{2}=\frac{D^n\phi(z)^\alpha}{\alpha^nz^\alpha}-\lambda+\frac{1}{2}
\end{multline*}
takes values in the convex hull of the image of $E$ under $\tfrac{D^nf(z)^\alpha}{\alpha^nz^\alpha}$, hence
\[
Re\left(\frac{D^n\phi(z)^\alpha}{\alpha^nz^\alpha}-\lambda+\frac{1}{2}\right)>\beta,
\]
so that
\[
Re\frac{D^n\phi(z)^\alpha}{\alpha^nz^\alpha}>\beta+\lambda-\frac{1}{2}.
\]
That is $\phi\in T_n^\alpha(\beta+\lambda-\tfrac{1}{2})$.
\end{proof}

\begin{corollary}
Let $f\in T_n^\alpha(\beta)$, $g\in T_0^\alpha(\lambda)$. Then $\phi\in T_n^\alpha(\beta+\lambda-\tfrac{1}{2})$.
\end{corollary}

\begin{corollary}
Let $f\in T_n^\alpha(\beta)$, $g\in T_0^\alpha(\tfrac{1}{2})$. Then $\phi\in T_n^\alpha(\beta)$.
\end{corollary}

For $0<\alpha\leq 1$ and $m\geq 1$, Theorem 3 can be improved as follows:

\begin{theorem}
Let $f\in T_n^\alpha(\beta)$ and $g\in T_m^\alpha(\lambda)$. Then for $0<\alpha\leq 1$ and $m\geq 1$ we have $\phi\in T_n^\alpha(\beta+\tfrac{\lambda}{2})\subset T_n^\alpha(\beta)$, that is
\[
Re\frac{D^n\phi(z)^\alpha}{\alpha^nz^\alpha}>\beta+\frac{\lambda}{2}.
\]
\end{theorem}

\begin{proof}
Consider the sequence $\{c_k\}_{k=0}^\infty$ given as
\[
c_k=\left\{
\begin{array}{ll}
1&\mbox{if $k=0$},\\
\frac{2\alpha^m}{(\alpha+k)^m}&\mbox{if $k\geq 1$.}
\end{array}
\right.
\]
It is easily verified that the sequence is also convex null if $0<\alpha\leq 1$ and $m\geq 1$. Then we define $\varphi(z)$ by
\[
\frac{\varphi(z)^\alpha}{z^\alpha}=1+2\sum_{k=2}^\infty\frac{\alpha^m}{(\alpha+k-1)^m}z^{k-1},
\]
and following the same arguement as in Theorem 3, we have the result.
\end{proof}

\begin{corollary}
Let $f\in T_n^\alpha(\beta)$, $g\in T_1^\alpha(\lambda)$. Then for $0<\alpha\leq 1$, $\phi\in T_n^\alpha(\beta+\tfrac{\lambda}{2})$.
\end{corollary}

\begin{corollary}
Let $f$, $g\in T_n^\alpha(\beta)$. Then for $0<\alpha\leq 1$ and $n\geq 1$, $\phi\in T_n^\alpha(\beta)$.
\end{corollary}

\begin{remark}

The quasi-convolution of $T_n^\alpha(\beta)$-functions is univalent in the unit disk if $n\geq 1$ and:

{\rm(i)} $\beta\geq\tfrac{1}{4}$ or
 
{\rm(ii)} $0<\alpha\leq 1$.

Thus our results provide an abundant source of functions which are univalent in the unit disk.
\end{remark}

\section{Applications}

For some applications of our results, let $f\in \mathbb{A}$ and define the following integrals:

\[
\phi_1(z)^\alpha=\frac{\alpha+c}{z^c}\int_0^zt^{c-1}f(t)^\alpha dt,\;\;\;\alpha+c>0,
\]
\[
\phi_2(z)^\alpha=\frac{2^\sigma}{z\Gamma(\sigma)}\int_0^z\left(\log\frac{z}{t}\right)^{\sigma-1}f(t)^\alpha dt,\;\;\;(\sigma>0)
\]
and
\[
\phi_3(z)^\alpha=\binom{\sigma+\gamma}{\gamma}\frac{\sigma}{z^\gamma}\int_0^z\left(1-\frac{t}{z}\right)^{\sigma-1}t^{\gamma-1}f(t)^\alpha dt,\;\;\;(\sigma,\;\gamma>0).
\]

The integral $\phi_1$ and its special cases ($\alpha=1$; $\alpha=1$, $c=0$ and $\alpha=1$, $c=1$) are well known and have been studied repeatedly in many literatures \cite{SA, BO, SDB, JKS, TOO, SO, RS}. The integrals $\phi_2$ and $\phi_3$  are new generalizations of the Jung-Kim-Srivastava one-parameter families of integral operators (for $\gamma>0$) \cite{JKS}. If for $\alpha>0$, we write $f(z)^\alpha=z^\alpha+A_2(\alpha)z^\alpha+...$, then in series form:
\[
\phi_1(z)^\alpha=z^\alpha+\sum_{k=2}^\infty\left(\frac{\alpha+c}{\alpha+c+k}\right)A_k(\alpha)z^{\alpha+k-1},
\]
Evaluating $\phi_2$ and $\phi_3$ in terms of Beta and Gamma functions we obtain 
\[
\phi_2(z)^\alpha=z^\alpha+\sum_{k=2}^\infty\left(\frac{2}{k+1}\right)^\sigma A_k(\alpha)z^{\alpha+k-1}
\]
and
\[
\phi_3(z)^\alpha=z^\alpha+\frac{\Gamma(\sigma+\gamma+1)}{\Gamma(\gamma+1)}\sum_{k=2}^\infty\frac{\Gamma(\gamma+k)}{\Gamma(\sigma+\gamma+k)}A_k(\alpha)z^{\alpha+k-1}.
\]

We shall now apply the quasi-convolution to prove:

\begin{theorem}
Let $f\in T_n^\alpha(\beta)$. Then $\phi_j\in T_n^\alpha(\beta)$, $j=1$, $2$, $3$.
\end{theorem}

Meanwhile let us prove the following:

\begin{lemma}
Let $p(z)=1+c_1z+c_2z^2+...$ be analytic in $E$ with Re $p(z)>\beta$, $0\leq\beta<1$. Then the real parts of the integral transformations
\[
\begin{split}
\vartheta(p(z))& =\frac{2^\sigma}{z\Gamma(\sigma)}\int_0^z\left(\log\frac{z}{t}\right)^{\sigma-1}p(t)dt,\;\;\;(\sigma>0)\\
& =2^\sigma+\sum_{k=1}^\infty\left(\frac{2}{k+1}\right)^\sigma c_kz^k
\end{split}
\]
and
\[
\begin{split}
\vartheta(p(z))& =\binom{\sigma+\gamma}{\gamma}\frac{\sigma}{z^\gamma}\int_0^z\left(1-\frac{t}{z}\right)^{\sigma-1}t^{\gamma-1}p(t)dt,\;\;\;(\sigma,\;\gamma>0)\\
& =\frac{\sigma+\gamma}{\gamma}+\frac{\Gamma(\sigma+\gamma+1)}{\Gamma(\gamma+1)}\sum_{k=1}^\infty\frac{\Gamma(\gamma+k)}{\Gamma(\sigma+\gamma+k)}c_kz^k
\end{split}
\]
are also greater than $\beta$.
\end{lemma}
\begin{proof}
The proofs of the assertions are similar.  We prove the second part as follows. Let $z=re^{i\theta}$ and $t=\rho e^{i\theta}$, $0<\rho\leq r<1$ so that
\[
\vartheta(p(re^{i\theta}))=\binom{\sigma+\gamma}{\gamma}\frac{\sigma}{r^\gamma}\int_0^r\left(1-\frac{\rho}{r}\right)^{\sigma-1}\rho^{\gamma-1}p(\rho e^{i\theta})d \rho,
\]
which gives
\[
Re\;\vartheta(p(re^{i\theta}))=\binom{\sigma+\gamma}{\gamma}\frac{\sigma}{r^\gamma}\int_0^r\left(1-\frac{\rho}{r}\right)^{\sigma-1}\rho^{\gamma-1}Re\;p(\rho e^{i\theta})d \rho.
\]
Since Re $p(z)>\beta$, we have
\[
Re\;\vartheta(p(re^{i\theta}))>\beta\binom{\sigma+\gamma}{\gamma}\frac{\sigma}{r^\gamma}\int_0^r\left(1-\frac{\rho}{r}\right)^{\sigma-1}\rho^{\gamma-1}d \rho.
\]
Evaluating the above integral in terms of Beta and Gamma functions, noting that
\[
\binom{\sigma}{\gamma}=\frac{\Gamma(\sigma+1)}{\Gamma(\sigma-\gamma+1)\Gamma(\gamma+1)}
\]
we shall obtain
\[
Re\;\vartheta(p(re^{i\theta}))>\beta\frac{\sigma+\gamma}{\gamma}
\]
which completes the proof.
\end{proof}

Now define
\[
g_1(z)^\alpha=\frac{\alpha+c}{z^c}\int_0^zt^{\alpha+c-1}\frac{1}{1-t}dt,
\]
\[
g_2(z)^\alpha=\frac{2^\sigma z^{\alpha-1}}{\Gamma(\sigma)}\int_0^z\left(\log\frac{z}{t}\right)^{\sigma-1}\frac{1}{1-t}dt
\]
and
\[
g_3(z)^\alpha=\binom{\sigma+\gamma}{\gamma}\frac{\sigma}{z^{\gamma-\alpha}}\int_0^z\left(1-\frac{t}{z}\right)^{\sigma-1}t^{\gamma-1}\frac{1}{1-t}dt.
\]
Then
\[
\frac{g_1(z)^\alpha}{z^\alpha}=\frac{\alpha+c}{z^{\alpha+c}}\int_0^zt^{\alpha+c-1}\frac{1}{1-t}dt
\]
is an integral iteration of $p(z)=1/(1-z)$ where Re $p(z)>1/2$. Hence $g_1\in T_0^\alpha(\tfrac{1}{2})$, that is Re $g_1^\alpha/z^\alpha>1/2$ (see \cite{BO}). Similarly by Lemma 4, Re $g_2^\alpha/z^\alpha>1/2$ and Re $g_3^\alpha/z^\alpha>1/2$. If we write $g_j^\alpha$, $j=1$, 2, 3 also in series form as $\phi_j^\alpha$ we see that for each $f\in T_n^\alpha(\beta)$, $\phi_j^\alpha=g_j^\alpha*f^\alpha$. Thus by Corollary 2, we find that the class $T_n^\alpha(\beta)$ is invariant under the transformations $\phi_j$, which proves Theorem 5.

This application in particular provides a new proof of the closure property of the class $T_n^\alpha(\beta)$ under integral tranformation $\phi_1$.
 
 \medskip

{\it Acknowledgements.} This work is in honour of my teacher and professor of Mathematics, Prof. E. A. Akinrelere, who retired recently from active service from the Obafemi Awolowo University, Ile-Ife, Nigeria. The author appreciates the professor's immense role in his academic development. The author is indebted to the referee for helpful comments and suggestions.

\vspace{10pt}

\hspace{-4mm}{\small{Received}}

\vspace{-12pt}
\ \hfill \
\begin{tabular}{c}
{\small\em  Department of Mathematics}\\
{\small\em  University of Ilorin}\\
{\small\em  Ilorin, Nigeria}\\
{\small\em E-mail: {\tt kobabalola@gmail.com}}\\
\end{tabular}

\end{document}